\documentclass[12pt, reqno]{amsart}

\usepackage{amssymb, fullpage, amsmath}
\usepackage[colorlinks]{hyperref}
\hypersetup{citecolor=blue}
\usepackage[alphabetic]{amsrefs}
\numberwithin{equation}{section}
\theoremstyle{plain}
\newtheorem{thm}{Theorem}[section]
\newtheorem{cor}{Corollary}[section]

\newtheorem{prop}{Proposition}[section]

\theoremstyle{definition}
\newtheorem{definition}{Definition}[section]

\newtheorem{remark}{Remark}[section]
\newtheorem{example}{Example}[section]

\newcommand{\qede}{\hfill $\diamond$}

\newcommand{\E}{\mathcal E}
\newcommand{\R}{\mathbb R}
\newcommand{\Q}{\mathbb Q}
\newcommand{\N}{\mathbb N}
\newcommand{\Z}{\mathbb Z}

\newcommand{\conv}{\mathrm{conv}}
\newcommand{\dist}{\mathrm{dist}}
\renewcommand{\int}{\mathrm{int}}
\newcommand{\eps}{\varepsilon}

\begin{document}

\title{Some remarks on convex analysis in topological groups}

\dedicatory{Dedicated to the memory of Jean-Jacques Moreau}

\author{Jonathan M. Borwein \and Ohad Giladi}
\address{Centre for Computer-assisted Research Mathematics and its Applications (CARMA), School of Mathematical and Physical Sciences, University of Newcastle, Callaghan, NSW 2308, Australia}
\email{jonathan.borwein@newcastle.edu.au, ohad.giladi@newcastle.edu.au}

\thanks{This research is supported in part by The Australian Research Council}

\begin{abstract}
We discuss some  key results from convex analysis in the setting of topological groups and monoids. These include separation theorems, Krein-Milman type theorems, and minimax theorems.
\end{abstract}

\maketitle

\section{Introduction}

\subsection{Background}
A topological group is a group which is also a topological space, such that the group operations are continuous. In this note we consider only commutative groups. Similarly, a topological monoid is a monoid (i.e., commutative semigroup with unit), which is also a topological space, such that the addition operation is continuous. In \cite{BG15} the present authors proposed a natural convexity structure for  groups and monoids that coincides with the classical notion when the underlying structure is a  vector space. It is then natural to ask when known algebraic or topological  results for  vector spaces still hold true in a  group or monoid. We should note that if a semigroup does not have a natural identity we simply add one.

Earlier related work is to be found in  in~\cite{BG15, CT96, Mor63, Mor63Book, Kin93, Par10, Pon14, XXC13}, among other authors.  
It is appropriate to point out that Moreau~\cite{Mor63, Mor63Book} studies the \emph{infimal convolution} in a monoid. For extended real-valued functions $f$ and $g$  he defines the inf-convolution by
\begin{align}\label{def inf} f \,\square\, g (x) = \inf_{y+z=x} \big[f(y)+g(z)\big],\end{align}
and  observes that for subadditive functions, the monoid provides the appropriate level of generality wherein to study infimal convolution.

Our own motivation is discussed in \cite{BG15} where also various  illustrative examples are given and  which provided a variety of primarily algebraic results.

In this note we study two topological topics. First, we look at topological separation theorems and their consequences, including a group-theoretic version of the Krein-Milman theorem and Milman's converse theorem. To prove adequate separation theorems, we define and study a group version of the well-known gauge functional. We show that in many respects it behaves similarly to the case of locally convex topological vector spaces. (See Section~\ref{sec separate}.) Then in Section~\ref{sec km} we use the separation results to prove a version of the Krein-Milman theorem for locally convex topological groups. We note that some versions of the Krein-Milman theorem have also been studied in the case of topological monoids/lattice structures, e.g., in~\cite{Pon14}.

Second, we look at the classical minimax theorem. The first proof was due to Von-Neumann~\cite{Neu28}, and later  generalisations and different proofs appeared in~\cite{B15, Fan53,  Sio58,Kin05,BZ86} and elsewhere. Herein we show that using the results of~\cite{Fan53}, one can easily deduce a satisfactory minimax theorem for appropriate topological monoids. (See Section~\ref{sec minimax}.)

\subsection{Locally convex topological groups basics}
We begin with some basic definitions. For more about topological groups, see for example~\cite{HR79}.

\begin{definition}[Topological group]
A group endowed with a topologiy is said to be a topological group if the group operations are continuous. That is, the function $(x,y)\mapsto x-y$ is continuous. 
\end{definition}

We also require a definition of a topological monoid.

\begin{definition}[Topological monoid]
A monoid which is also a topological space is said to be a topological monoid if the addition operation is continuous.
\end{definition}

While the notion of convexity is usually studied in the context of vector spaces, it can be defined and studied in a very general setting. We refer the reader to ~\cite{Vel93} for more about abstract convexity, and to~\cite{BG15} for more about convexity in groups and monoids. In particular, given a space with an abstract collection of convex sets, we can define the following notion.

\begin{definition}[Locally convex topological space]
A topological space is said to be locally convex if its topology has a basis which contains only convex sets.
\end{definition}

In a topological group $X$ we have that for every $x_0\in X$, the map $x\mapsto x+x_0$ as well as its inverse $x\mapsto x-x_0$ are continuous (it suffices to assume only the latter). In particular, it follows that if $U$ is a neighbourhood of $x\in X$, then by the continuity, $U-x$ is a neighbourhood of 0. Thus, any neighbourhood of $x$ can be written as $x+U$, where $U$ is a neighbourhood of 0. Note that this is \emph{not} the case for arbitrary topological monoids. This is evident  if we consider only  the simple example $X= \R$ with the operation $x\wedge y$, that is taking the minimum. 

It is known that if $X$ is a topological group and the topology is Hausdorff, then singletons are closed sets. Indeed,  in topological groups, the $T_1$ and Hausdorff  properties are equivalent. In this note, for these and other reasons all topological groups will assumed to be Hausdorff. 

By the maximum formula~\cite[Theorem 3]{BG15}, it follows that every finite convex function on a semidivisible group is equal to the supremum over its additive minorants. Thus, locally convex topological groups admit `many' additive functions. This is in contrast to the non-locally convex case, for example in the  topological vector space $L_p([0,1])$, $p\in (0,1)$.

\begin{prop}\label{prop singleton}
Assume that $X$ is a locally convex, $T_1$ topological group. Then all singletons are convex, and no elements have finite order. In particular, the group has at most unique divisors.
\end{prop}

 By considering the discrete topology, the previous result implies that in a locally convex group points are convex (resp. closed) iff the topology is Hausdorff.

\begin{proof}[Proof of Proposition~\ref{prop singleton}]
Let $x\in X$. If $y\neq x$ then since the topology is $T_1$, there exists $U$ open and convex, such that $x\in U$, $y\notin U$. Since $\conv(\{x\}) \subseteq U$, the first assertion follows. To prove the second assertion, suppose that $x \neq 0$, and note that if $C\subseteq X$ is convex, $x\in C$ and $x$ is of finite order, then there exists $m\in \N$ such that $mx = 0 = m\cdot 0$. Thus, $0\in \conv(\{x\})$, which can happen only if $\{x\}$ is not convex.
\end{proof}

Also, recall the following definitions.

\begin{definition}[Semidivisible monoid]
A monoid $X$ is said to be $p$-semidivisible if there exists $p\in \N$ prime such that $pX=X$. That is, for every $x\in X$ there exists $y\in X$ such that $x=py$. 
A monoid is said to be divisible if it $p$-semidivisible for every $p\in \N$ prime. 
\end{definition}

Note that if $X$ is $p$-divisible and $q$-divisible  then it is $p^n q^m$-divisible for $m,n \in \N$.

\begin{definition}[Uniquely divisible monoid] 
A monoid $X$ is said to be uniquely divisible if it is divisible and for every $n \in \N$, the map $x\mapsto nx$ is injective.
\end{definition}

It is known that torsion-free divisible abelian groups are modules over $\Q$ \cite{BG15}, and so the class of merely semidivisible monoids and groups is a much larger and potentially richer one. We recall the following example of a semidivisible  groups which is not divisible. These examples appeared already in~\cite{BG15}, but now they can be usefully considered in the context of locally convex topological groups.

\begin{example}[$\sigma$-algebra with symmetric difference and a measure as a distance]\label{ex symm measure}
Let $(\Omega,\mathcal F,\mu)$ be a measure space, that is, $\Omega$ is a set, $\mathcal F$ is a $\sigma$-algebra of subsets of $\Omega$, and $\mu$ is a positive measure on elements in $\mathcal F$. For $A,B\in \mathcal F$, define $A+B = A\triangle B$. Then it is known (see for example~\cite{BG15}) that under this operation, $A+B=B+A$, $A+\emptyset = \emptyset+A = A$ and $A+A=\emptyset$. Also, it is known that if $\mathcal A \subseteq F$, then
\begin{align}\label{conv is sum}
\conv(\mathcal A) = \left\{A\subseteq X~\left| ~A = \sum_{i=1}^nA_i, ~~ A_i\in \mathcal A, ~~ n\in \N\right.\right\}.
\end{align}
Also, for $A,B\in \mathcal F$, define $d_{\mu}(A,B) = \mu(A\triangle B)$. Then it is known that $d_{\mu}(\cdot,\cdot)$ is a pseudo-metric on $\mathcal F$. Therefore, let $X = \mathcal F / \sim$, where $A\sim B \iff \mu(A\triangle B) = 0$. Assume that $A_n\to A$, $B_n\to B$ in $X$, that is $\mu(A_n\triangle A) \to 0$, $\mu(B_n\triangle B) \to 0$. Then we have $\mu\big((A_n\triangle B_n)\triangle (A\triangle B)\big) = \mu\big((A_n\triangle A)\triangle (B_n\triangle B)\big) \stackrel{(*)}{\le} \mu(A_n\triangle A)+\mu(B_n\triangle B) \to 0$, where in ($*$) we used the fact that for every sets $A,B$, we have $\mu(A\triangle B) = \mu\big((A\cup B)\setminus (A\cap B)\big) \le \mu(A\cup B) \le \mu(A)+\mu(B)$. Therefore, it follows that $A_n+B_n \to A+B$. Since $B_n=-B_n$ and $B=-B$, we also have $A_n-B_n \to A-B$, which shows that $X$ is indeed a topological group. On the other hand, in general $X$ is not locally convex. To see this, consider the example where $\Omega = [0,1]$, $\mathcal F = \mathcal B([0,1])$, that is, the Borell sets on $[0,1]$ and $\mu$ is the Lebesgue measure. Then for $\eps>0$, the set $\mathcal A_\eps = \big\{A\in \mathcal F~\big|~\mu(A)<\eps\big\}$ is a neighbourhood of $\emptyset$. However, by~\eqref{conv is sum}, it follows that if we choose $A_i = \big((i-1)\eps/2, i\eps/2\big)$, $1\le i \le n$, where $n = \left\lfloor \frac \eps 2\right\rfloor +1$. Note that the sets $A_i$ are disjoint. Thus, we have
\begin{align*}
A = \sum_{i=1}A_i \stackrel{(*)}{=} \bigcup_{i=1}^nA_i = [0,1],
\end{align*}
where in ($*$) we used the fact that for disjoint sets $A_1,\dots,A_n$, we have $\sum_{i=1}^nA_i = \bigcup_{i=1}^nA_i$. In particular, we have $[0,1]\in \conv(\mathcal A_\eps)$ but since $[0,1]\notin \mathcal A_\eps$, it follows that $\mathcal A_\eps$ is not convex, and so $X$ is not locally convex.

The group $X$ is connected. Indeed, let $\mathcal A\subseteq X$ be the connected component that contains $\emptyset$. In particular, $\mathcal A$ is open. Taking 
\[\mathcal A_\eps = \big\{B\in\mathcal F~\big|~ \mu(A\triangle B) <\eps \text{ for some }~ A\in \mathcal A\big\},\]
since $d_{\mu}(\cdot,\cdot)$ is a distance on $X$, it follows that $\mathcal A_\eps$ is open. Hence, we must have $\mathcal A=X$ and so $X$ is connected.
\qede
\end{example}

\begin{example}[Positive hyperbolic group]\label{ex hyper}
Let $X$ be the commutative group of matrices of the form $M(\theta)$, where $\theta\in \R$, and 
\[M(\theta) = \left[\begin{array}{cc}\cosh(\theta) & \sinh(\theta) \\ \sinh(\theta) & \cosh(\theta)\end{array}\right].\]
This can be thought of the `positive' branch of the group $X_{\R}$, as defined in~\cite{BG15}. See also Remark~\ref{rmk hyper} below. The group operation is given by the matrix multiplication. It follows that we have $M(\theta_1)\cdot M(\theta_2) = M(\theta_1+\theta_2)$. The topology on $X$ is the topology induced by the euclidean metric in $\R^4$. The function $\theta \mapsto M(\theta)$ is continuous. Thus, the group $X$ is connected. Note also that this group is divisible, as we have $M(\theta) = \big(M(\theta/n)\big)^n$ for every $\theta\in \R$ and every $n\in \N$. $X$ is also locally convex, since if $x=M(\theta)\in X$, let 
\[U(\theta,\eps) = \big\{M(\theta')~\big|~|\theta'-\theta|<\eps\big\},\] 
for some $\eps >0$. Then $U(\theta,\eps)$ is an open and convex neighbourhood of $x$. To show that $U(\theta,\eps)$ is open, let $M(\theta')\in U(\theta,\eps)$, and $M(\theta'')\in X$ such that 
\begin{align}\label{small dist c4}
\dist_{\mathbb R^4}\big(M(\theta'),M(\theta'')\big) \le \eps'.
\end{align}
We would like to show that if $\eps'$ is sufficiently small, $M(\theta'')\in U(\theta,\eps)$. Indeed,~\eqref{small dist c4} implies in particular that 
\begin{align}\label{diff small}
\big|\sinh(\theta')-\sinh(\theta'')\big| \le \eps'.
\end{align} and so we have
If $\eps'$ is sufficiently small, then since $\sinh(\cdot)$ is continuous and injective, we have that $|\theta''-\theta|<\eps$. Altogether, we have $M(\theta'')\in U(\theta,\eps)$ and so $U(\theta,\eps)$ is open. To show that $U(\theta,\eps)$ is convex, let $M(\theta_1),\dots,M(\theta_n)\in U(\theta,\eps)$, $m_1,\dots,m_n\in \N$ and assume 
\begin{align}\label{conv condition}
m\big(M(\theta)\big) = \sum_{i=1}^nm_iM(\theta_i), ~~m=\sum_{i=1}^nm_i.
\end{align} 
Now, if we have that $M(\theta) = M(\theta')$ then by comparing all the entries of the two matrices, it follows that $\theta=\theta'$. Hence,~\eqref{conv condition} implies that $m\theta = \sum_{j=1}^nm_j\theta_j$. Since $\theta_1,\dots,\theta_n\in (\theta'-\eps, \theta'+\eps)$ and $(\theta'-\eps, \theta'+\eps)$ is a convex subset of $\R$, it follows that $\theta \in (\theta'-\eps, \theta'+\eps)$, which proves that $U(\theta,\eps)$ is convex. Given any open neighbourhood $U$ of $M(\theta)$, then again since the topology on $X$ is the topology induced by the metric in $\R^4$ and since $\sinh(\cdot)$ and $\cosh(\cdot)$ are continuous, there exists $\eps>0$ such that $U(\theta,\eps) \subseteq U$. This shows that the group $X$ is locally convex. 
\qede
\end{example}

\begin{remark}\label{rmk hyper}
If we consider the group
\[X_{\R} = \Big\{e^{it}M(\theta)~\Big|~ \theta,t \in \R\Big\},\]
where $M(\theta)$ is as in Example~\ref{ex hyper}. This is a group under the matrix multiplication. See~\cite{BG15} for the details. Let the topology on $X_{\R}$ be the topology induced by the euclidean metric on $\mathbb C^4$. Then $X_{\R}$ is connected as it is the image of the continuous map $(t,\theta)\mapsto e^{it}M(\theta)$. On the other hand, $X_{\R}$ is not locally convex. To see this, let $U$ be an open neighbourhood of $M(0)$. Since the topology is the topology on $\mathbb C^4$, there must be $\eps>0$ such that $\big\{e^{it}~\big|~|t|<\eps\big\}\subseteq U$. However, we have
\begin{align}\label{circle not convex}
\mathrm{conv}\left(\big\{e^{it}~\big|~|t|<\eps\big\}\right) = \big\{e^{it}~\big|~ t\in [0,2\pi)\big\}.
\end{align}
See~\cite{BG15} for a more detailed study of convexity in the circle group. In particular,~\eqref{circle not convex} implies that there is no $U'$ convex and open such that $M(0)\in U'$ and $U'\subseteq U$. Hence $X_{\R}$ is not locally convex.
\end{remark}

\section{Separation theorems in groups and monoids}\label{sec separate}

\subsection{Convexity in groups and monoids}
For the sake of completeness, we present some basic facts that appeared in~\cite{BG15}. First, we define convex sets in monoids.

\begin{definition}[Convex set]
Let $X$ be a monoid and $A\subseteq X$. $A$ is said to be convex, if for every $x_1,\dots,x_n\in A$, every $m_1,\dots,m_n\in \N=\{1,2,\ldots\}$ such that $\sum_{i=1}^nm_i=m$, we have
\begin{align*}
mx = \sum_{i=1}^nm_ix_i \Longrightarrow x\in A.
\end{align*}
\end{definition}

Next, we define some classes of functions on monoids.

\begin{definition}[Convex and concave functions]
Let $X$ be a monoid. A function $f:X\to[-\infty,\infty]$ is said to be convex if for every $x_1,\dots,x_n\in X$, $m_1,\dots,m_n\in \N$ such that $m=\sum_{i=1}^nm_i$ and $mx=\sum_{i=1}^nm_ix_i$, we have 
\begin{align*}
mf(x) \le \sum_{i=1}^nm_if(x_i).
\end{align*}
A function $f:X\to [-\infty,\infty]$ is said to be concave if the function $-f$ is convex.
\end{definition}

Here and in what follows, we let $\infty-\infty = +\infty$ when considering a convex function and $\infty-\infty = -\infty$ when considering a concave function.

\begin{definition}[Generalised affine functions]
Let $X$ be a monoid. A function $f:X\to[-\infty,\infty]$ is said to be affine if it is both convex and concave.
\end{definition}

\begin{definition}[Subadditive functions]
Let $X$ be a monoid. A function $f:X\to[-\infty,\infty]$ is said to be subadditive if for every $x,y\in X$, we have
\begin{align*}
f(x+y) \le f(x)+f(y).
\end{align*}
\end{definition}

\begin{definition}[$\N$-sublinear functions]
Let $X$ be a monoid. A function $f:X\to[-\infty,\infty]$ is said to be $\N$-sublinear if it is subadditive and in addition it is positively homogeneous, i.e.,
\begin{align*}
f(kx) = kf(x), ~~ k\in \N\cup\{0\},~ x\in X.
\end{align*}
\end{definition}

Note that all the classes of functions defined above can be defined when $X$ is a group rather than a monoid. More generally, the above classes can be defined when $X$, as well as the range are semimodules. For the sake of concreteness, we do not include the most general case. See~\cite{BG15} for the definitions in  generality. We now study some of the properties of the inf-convolution, defined in~\eqref{def inf}.

\begin{prop}[Inf-convolution of subadditive functions, Moreau]\label{prop inf subadd}
Let $X$ be a monoid and assume that $f,g:X\to[-\infty,\infty]$ are subadditive. Then $f\,\square\,g$ is subadditive.
\end{prop}

\begin{proof}
Let $x_1,x_2\in X$. Assume that we can write $x_1=y_1+z_1$, $x_2=y_2+z_2$. This is always possible since we can choose one of the elements to be $0$. Thus we have $x_1+x_2 = (y_1+y_2)+(z_1+z_2)$, and so
\begin{align}\label{ineq subadd}
\nonumber f\,\square \, g(x_1+x_2) &\le f(y_1+y_2)+g(z_1+z_2) 
\\ & \le \big[f(y_1)+g(z_1)\big]+\big[f(y_2)+g(z_2)\big].
\end{align}
Taking the infimum over the right side of~\eqref{ineq subadd}, the result follows.
\end{proof}

Under the assumption that $X$ is semidivisible, we can obtain much stronger convexity results regarding $f\,\square\, g$.

\begin{prop}[Inf-convolution of convex and $\N$-sublinear functions]
Let $X$ be a $p$-semidivisible monoid and assume that $f,g:X\to [-\infty,\infty]$ are convex (resp. $\N$-sublinear). Suppose, moreover, that $X $ has at most unique divisors  as holds in the locally convex case.
Then $f\,\square\, g$ is convex (resp. $\N$-sublinear).
\end{prop}

\begin{proof}
Assume first that $f$ and $g$ are convex. Let$x, x_1,\dots,x_n\in X$ and $m_1,\dots,m_n\in \N$ such that $p^kx = \sum_{i=1}^nm_ix_i$ and $p^k=\sum_{i=1}^nm_i$ where $k\in\N$. Let $y_1,\dots,y_n,z_1\dots,z_n\in X$ be such that $x_i=y_i+z_i$, $1\le i \le n$. Since $X$ is $p$-semidivisible, there exist $y,z\in X$ such that $p^ky = \sum_{i=1}^nm_iy_i$ and $p^kz=\sum_{i=1}^nm_iz_i$. Thus, we have $p^kx=p^k(y+z)$. Since $X$ is uniquely divisible, we have $x=y+z$. Thus,
\begin{eqnarray}\label{ineq conv}
\nonumber p^kf\,\square\, g(x) & \le & p^kf(y)+p^kg(z)
\\ \nonumber & \stackrel{(*)}{\le} & \sum_{i=1}^nm_if(y_i)+\sum_{i=1}^nm_ig(z_i)
\\ & = & \sum_{i=1}^nm_i\big[f(y_i)+g(z_i)\big],
\end{eqnarray}
where in ($*$) we used the convexity of $f$ and $g$. Taking the infimum over the right side of~\eqref{ineq conv}, it follows that 
\begin{align*}
p^kf\,\square \, g(x) \le \sum_{i=1}^nm_if\,\square\, g(x_i).
\end{align*}
Now, use~\cite[Proposition 6]{BG15} to deduce that $f\,\square \, g$ is convex. To prove the $\N$-sublinear case, using Proposition~\ref{prop inf subadd} it is enough to prove that $f\,\square\, g$ is positively homogeneous. Assume that $px = y'+z'$. Then since $X$ is $p$-semidivisble, there exist $y,z\in X$ such that $py=y'$ and $pz=z'$. This means that $px=p(x+y)$. Since $X$ is uniquely divisible, it follows that $x=y+z$. Therefore, we have,
\begin{align*}
f\,\square\, g(px) & = \inf_{y'+z' = px}\big[f(y')+g(z')\big]
\\ & = \inf_{py+pz = px}\big[f(py)+g(pz)\big]
\\ & = p\inf_{py+pz = px}\big[f(y)+g(z)\big]
\\ & = p\inf_{y+z = x}\big[f(y)+g(z)\big].
\end{align*}
Now, using~\cite[Proposition 7]{BG15}, the result follows.
\end{proof}

We turn to the study of the gauge function.

\subsection{Rational dilation of sets and the Minkowski functional}

Given a set $A\subseteq X$ and $m\in \N$, let
\begin{align}\label{sum set}
\nonumber mA & = \left\{ x\in X~\left|~ x = \sum_{i=1}^mx_i, ~x_i\in A\right.\right\}
\\ & = \left\{ x\in X~\left|~ x  = \sum_{i=1}^nm_ix_i, ~x_i\in A, ~\sum_{i=1}^nm_i = m\right.\right\} .
\end{align}
In the case of convex sets, we can generalise~\eqref{sum set} in the following way.
\begin{definition}[Rational dilation of set]\label{def dilation}
Let $X$ be a monoid and $C\subseteq X$ be a convex set. Also, let $q\in \Q_+$ be a reduced fraction. Define
\begin{align*}
qC = \left\{ x\in C~\left|~ lx = \sum_{i=1}^nm_ix_i,~ x_i\in C,~ \sum_{i=1}^nm_i=m,~ \frac m l = q\right.\right\}.
\end{align*}
\end{definition}

\begin{remark}
Note that if $q=k\in \N$ then Definition~\ref{def dilation} coincides with~\eqref{sum set}. Also, note that if $C$ is not convex, we do not necessarily have $1\,C = C$. \qede
\end{remark}

We have the following proposition.

\begin{prop}[Dilations of convex sets are monotone]\label{prop monotone}
Assume that $X$ is a monoid and $C\subseteq X$ is convex and $0\in C$. Assume that $q_1,q_2\in \Q_+$ are reduced fractions with $q_1\le q_2$. Then $q_1C\subseteq q_2C$.
\end{prop}

\begin{proof}
Write $q_1 = \frac m l $ and $q_2 = \frac{m'}{l'}$. Since $q_1\le q_2$, we have that $ml' \le m'l$. Assume that $x\in q_1C$. Then we can write $lx = \sum_{i=1}^nm_ix_i$, $x_i\in C$, $\sum_{i=1}^nm_i = m$. Thus, it follows that $ll'x = \sum_{i=1}^nm_il'x_i = \sum_{i=1}^nm_il'x_i +(m'l-ml')\cdot 0$. Now, we have $\sum_{i=1}^nm_il' = ml'$, and so $\sum_{i=1}^nm_il'+(m'l-ml') = m'l$. Since $\frac{m'l}{ll'} = \frac{m'}{l'}$, we have $x\in \frac{m'}{l'}C$, which completes the proof.
\end{proof}

The first step in our proof of the Hahn-Banach separation theorem requires us  to construct a group version of the Minkowski functional. For this we need the following result, which is an immediate consequence of Proposition~\ref{prop monotone}.

\begin{cor}\label{cor interval}
Assume that $X$ is a monoid, $C\subseteq X$ is convex and $0\in C$, and let $x\in C$. Then the set $\{q\in \Q_+~|~ x\in qC\}$ is of the form $[\lambda,\infty)\cap\Q_+$ or $(\lambda,\infty)\cap\Q_+$, where $\lambda \in \R_+$.
\end{cor}

Using Corollary~\ref{cor interval}, it is natural to define the following.
\begin{definition}[Minkowski functional for groups]
Let $X$ be a monoid and $C\subseteq X$. Define
\begin{align}\label{def mink}
\rho_C(x) = \inf\big\{q\in \Q_+~\big| ~ x\in qC\big\}.
\end{align}
If there is no $q\in \Q_+$ that satisfies~\eqref{def mink}, define $\rho_C(x) = \infty$.
\end{definition}

\begin{remark}
Note that Proposition~\ref{prop monotone}, and consequently Corollary~\ref{cor interval} and Definition~\ref{def mink}, are purely algebraic, and do not require any topological structure. \qede
\end{remark}

\begin{remark}\label{rmk threshold}
By Definition~\ref{def mink}, if $x\in C$ then $\rho_C(x) \le 1$, and if $x\notin C$, then $\rho_C(x) \ge 1$. \qede
\end{remark}

\begin{prop}[Sublinearity of the Minkowski functional]\label{prop mink}
Assume that $X$ is a monoid, and $C\subseteq X$ is convex. Then the functional defined by~\eqref{def mink} is $\N$-sublinear.
\end{prop}

\begin{proof}
We start by showing that $\rho_C$ is subadditive. Indeed, let $x,y\in X$, and choose $m,m',l,l'\in \N$ such that $\frac{m}{l}<\rho_C(x)+\eps$, $\frac{m'}{l'} \le \rho_C(y)+\eps$, and $lx =\sum_{i=1}^nm_ic_i$, $c_i\in C$, $\sum_{i=1}^nm_i = m$ and $l'y=\sum_{i=1}^{n'}m_i'c_i'$, $c_i'\in C$, $\sum_{i=1}^{n'}m_i'=m'$. Thus, we have $ll'(x+y) = \sum_{i=1}^nl'm_ic_i+\sum_{i=1}^{n'}lm_i'c_i'$, and by~\eqref{def mink} we have
\begin{align*}
\rho_C(x+y) \le \frac 1 {ll'}\left(\sum_{i=1}^nl'm_i+\sum_{i=1}^{n'}lm_i'\right) = \frac{l'm+lm'}{ll'} = \frac m l + \frac{m'}{l'} \le \rho_C(x)+\rho_C(y)+2\eps.
\end{align*}
Since $\eps>0$ is arbitrary, the subadditivity of $\rho_C$ follows. To show the positive homogeneity, note that since $\rho_C$ is subadditive, we have $\rho_C(kx) \le k\rho_C(x)$ for all $x\in X$, $k\in \N$. Thus, all we need to prove is $\rho_C(kx) \ge k\rho_C(x)$. Indeed,
\begin{eqnarray*}
\rho_C(kx) & \stackrel{\eqref{def mink}}{=} & \inf\left\{\left.\frac m l~\right|~lkx =\sum_{i=1}^nm_ic_i,~~ \sum_{i=1}^nm_i=m,~~ c_i\in C\right\}
\\ & = & k\inf\left\{\left.\frac m {lk}~\right|~lkx =\sum_{i=1}^nm_ic_i,~~ \sum_{i=1}^nm_i=m,~~ c_i\in C\right\}
\\ & \stackrel{(*)}{\ge} & k\inf\left\{\left.\frac m {l}~\right|~lx =\sum_{i=1}^nm_ic_i,~~ \sum_{i=1}^nm_i=m,~~ c_i\in C\right\}
\\ & = & k\rho_C(x),
\end{eqnarray*}
where in ($*$) we used the fact that we  take an infimum over a larger set. Altogether, we have  that $\rho_C(kx)=k\rho_C(x)$, and along with the subadditivity of $\rho_C$, this completes the proof.
\end{proof}

\subsection{Hahn-Banach separation theorem}

Under no additional topological assumption, we can obtain the following group version of the Hahn-Banach separation theorem.

\begin{thm}[Hahn-Banach weak separation]\label{HB groups}
Assume that $X$ is a semidivisible topological group, and $C, D\subseteq X$ convex. Assume that $\int\, C \neq \emptyset$ and $D \cap \int\, C \neq \emptyset$. Then there exists $\varphi:X\to [-\infty,\infty]$ which is nonzero and affine such that
\[\sup_{c\in C}\varphi(c) \le \inf_{d\in D}\varphi(d).\]
\end{thm}

\begin{proof}
Assume without loss of generality that $0\in \int\, C$. Let $f = \rho_C$ and $g = \iota_{\overline {D}}-1$.  $f,g:X\to (-\infty,\infty]$ are convex and $-g \le f$. Then by~\cite[Theorem 2]{BG15}, there exists a nonzero affine $\varphi:X\to \R$ such that $-g\le \varphi \le f$. Now, for every $c\in C$, $\varphi(c) \le f(c) \le 1$ and for every $d\in D$, $\varphi(d) \ge -g(d) =1$, which completes the proof.
\end{proof}

Many applications of the Hahn-Banach separation theorem, require \emph{strict} separation:  if $C$ is convex and $x\notin C$ then there exists $\varphi$ linear (additive, in our case), such that $\sup_{c\in C}\varphi(c)<\varphi(x)$. In order to prove such a result in groups, we need more topological structure. Under additional topological assumptions we draw a stronger conclusion about $\rho_C$. In particular, we have the following proposition.

\begin{prop}\label{prop rho strict}
Assume that $X$ is a topological group, $C\subseteq X$ is convex, and $0\in \int \,C$. Then $\rho_C$ is everywhere continuous on its domain. If, in addition, $X$ is  connected, then $\rho_C$ is everywhere finite.
\end{prop}

\begin{proof}
Let $U$ be a neighbourhood of $0$. Then $V = \int \, C\cap U$ is also a neighbourhood of $0$. Now, $\frac 1 l V \subseteq \frac 1 l C$. Note that if we define $\phi_l(x) = lx$, then $\phi_l$ is continuous and $\frac 1 l V = \phi_l^{-1}(V)$, which implies that $\frac 1 l V$ is open. Also, since $0\in V$, we have that $0\in \frac 1 l V$. Thus, $V$ is again an open neighbourhood of $0$. Assume that $x\in \frac 1 l C$, then $lx\in C$ and then by the positive homogeneity of $\rho_C$, $l\rho_C(x) = \rho_C(lx) \le 1$. Thus $\rho_C(x) \le \frac 1 l$. This means that $\rho_C$ is continuous at $0$. Now, if $x_0\in X$, since we have $\rho_C(x)-\rho_C(x_0) \le \rho_C(x-x_0)$, and $\rho_C(x_0)-\rho_C(x) \le \rho_C(x_0-x)$, continuity at $x=0$ implies continuity everywhere else. If $X$ is connected, then $\rho_C^{-1}(\R)$ is both open and closed, and since it is not empty, it must be all of $X$ (see also Prop. \ref{thm cover}). This concludes the proof.
\end{proof}

The following is an easy but useful proposition.

\begin{prop}\label{prop subadd}
Let $X$ be a monoid. If $a:X\to [-\infty,\infty]$ is affine and everywhere finite, then we can write $a(x)= \alpha+\phi(x)$, where $\alpha\in \R$ and $\phi:X\to \R$ is additive. If $\alpha\ge 0$, then $a$ is also subadditive. 
\end{prop}

\begin{proof}
If $a$ is affine, then it is both convex and concave. Then $\phi(x) = a(x)-a(0)$ is convex, concave, and $\phi(0)=0$. Let $x_1,\dots,x_n\in X$. Then $n\sum_{i=1}^nx_i = \sum_{i=1}^n 1\cdot (nx_i)$. Thus, since $\phi$ is both convex and concave, we have
\begin{align}\label{equality sum}
n\phi\left(\sum_{i=1}^nx_i\right) = \sum_{i=1}^n\phi(nx_i).
\end{align}
Letting $x_2=\dots=x_n=0$, it follows that $\phi$ is positively homogeneous. Thus,~\eqref{equality sum} gives
\begin{align*}
n\phi\left(\sum_{i=1}^nx_i\right) = \sum_{i=1}^nn\phi(x_i).
\end{align*}
which implies that $\phi$ is additive. Choosing $\alpha=a(0)$, the first assertion follows. To prove the second assertion, note that if $\alpha\ge 0$, we have 
\begin{align*}
a(x+y) = \alpha+\phi(x+y) = \alpha+\phi(x)+\phi(y) \le 2\alpha+\phi(x)+\phi(y) = a(x)+a(y),
\end{align*}
which concludes the proof.
\end{proof}

Another useful auxiliary result is the following early subadditive separation theorem due to Kaufman.

\begin{thm}[Kaufman,~\cite{Kau66}]\label{thm kaufman}
Assume that $X$ is a monoid and $f,g:X\to [-\infty,\infty)$ are subadditive, and $-g\le f$. Then there exists a finite additive map $a$ such that $-g\le a \le f$.
\end{thm}

We are now in a position to prove a strict separation theorem.

\begin{thm}[Hahn-Banach strict separation]\label{HB groups strict}
Assume that $X$ is a connected, locally convex topological group, $C\subseteq X$ is closed and convex, while $x_0\notin C$. Then there exists a continuous additive function $\varphi : X\to \R$ such that 
\begin{align}\label{ineq strict}
\sup_{c\in C}\varphi(c) < \varphi(x_0).
\end{align}
\end{thm}

\begin{proof}
Assume without loss of generality that $x_0=0$. Since $C$ is closed, there exists a convex neighbourhood $U$ of $0$ such that $U\cap C = \emptyset$. Let $f = \rho_U$ and $g = \iota_C-1$. Then $f,g$ are convex and $-g\le f$. Thus, by~\cite[Theorem 2]{BG15}, there exists $a:X\to \R$ nonzero and affine such that $-g\le a \le f$. $a$ can be assumed to be everywhere finite because $\rho_U$ is everywhere finite (see~\cite[Corollary 3]{BG15}). Since $-g\le a \le f$, we have $a(0) \le 0$. Use Proposition~\ref{prop subadd} to write $a = \phi+\alpha$ where $\phi$ is additive and $\alpha\le 0$. 

Thus, again by Proposition~\ref{prop subadd}, $-a$ is subadditive. Since by Proposition~\ref{prop mink}, $\rho_U$ is subadditive, use Theorem~\ref{thm kaufman} to deduce the existence of an additive $\varphi:X\to \R$, such that $a\le \varphi \le \rho_U$. Since $a\ge g$ we have $\varphi\ge g$, and we have $\varphi(c) \ge 1$ for every $c\in C$ and $\varphi(0) = 0$, which proves~\eqref{ineq strict}. To prove the continuity of $\varphi$, note that since $\varphi$ is additive, $-\varphi(x) = \varphi(-x) \le \rho_U(-x) = \rho_U(x)$ and so $|\varphi(x)|\le \rho_U(x)$. Thus, we have $\varphi(x)-\varphi(y)| = |\varphi(x-y)| \le \rho_U(x-y)$. Since $\rho_U$ is continuous (in fact it is enough that $\rho_U$ is continuous at $x=0$), it follows that $\varphi$ is continuous. This concludes the proof.
\end{proof}

\begin{remark}\label{rmk sum}
In many classical topological vector space  proofs of the separation theorem, one deduces separation between sets from separation between a set and a point by applying the latter to the a set of the form $A-B$, where $A,B$ are convex. In general,  however,  if $A,B\subseteq X$ are convex subsets of a group, $A+B$ need not be convex. For example if $X=\Z^2$, and we choose $A= \{(0,1),(2,0)\}$, $B=\{(0,2),(1,0)\}$, then $A+B = \{(1,1),(2,2),(3,0),(0,3)\}$. Also, we have $3\cdot(1,2) = 2\cdot(0,3)+1\cdot(3,0)$ but $(1,2)\notin A+B$. 

On the other hand, if the group is divisible then convexity is preserved under taking sums of sets. Indeed, if $a_1,\dots,a_n\in A$, $b_1,\dots,b_n\in B$ and $m_1,\dots,m_n\in \N$ such that $\sum_{i=1}^nm_i = m$, then there exist $a,b\in X$ such that $ma = \sum_{i=1}^nm_ia_i$ and $mb=\sum_{i=1}^nm_ib_i$. Since $A,B$ are convex, we have $a\in A$ and $b\in B$. Thus, $\sum_{i=1}^nm_i(a_i+b_i) = m(a+b)$. Now, if we assume that $\sum_{i=1}^nm_i(a_i+b_i) = mx$, then if we assume further that we have unique divisibility (such as the case for locally convex topological groups, as shown in Proposition~\ref{prop singleton}), then $x=a+b\in A+B$, which proves that $A+B$ is indeed convex. \qede
\end{remark}

\begin{remark}
While $A+B$ need not be convex for convex $A,B\subseteq X$, as shown in Remark~\ref{rmk sum}, it is true that translations of convex sets are convex. Indeed, if $A\subseteq X$ is convex and $x_0\in X$ then $x_0+A$ is convex. To see this, let $a_1,\dots,a_n\in A$, $m_1,\dots,m_n\in \N$, $m=\sum_{i=1}^nm_i$, and assume $mx = \sum_{i=1}^nm_i(x_0+a_i) = mx_0+\sum_{i=1}^nm_ia_i$. Then $m(x-x_0) = \sum_{i=1}^nm_ia_i$. Since $A$ is convex, it follows that $x-x_0\in A$, which means that $x\in x_0+A$. Also, by~\cite[Proposition 3]{BG15}, if $T:Y\to X$ is additive and $A\subseteq X$ is convex, then $T^{-1}(A)$ is also convex in $Y$. \qede
\end{remark}

\begin{remark}
If $A$ is convex and $A+U$ is convex for every convex neighbourhood of $0$, then $\overline A$ is also convex. In particular, by Remark~\ref{rmk sum}, the closure of a convex set in a divisible locally convex topological group is convex. \qede
\end{remark}

\section{Extreme points in topological groups}\label{sec km}

\subsection{The Krein-Milman theorem in topological groups}

We begin with a few natural definitions.

\begin{definition}[Extreme points]\label{def extreme}
Let $X$ be a group and $A\subseteq X$. A point $x\in A$ is said to be an extreme point of $A$, if whenever $mx = \sum_{i=1}^nm_ix_i$, $m_i\in \N$, $\sum_{i=1}^nm_i = m$, $x_i\in A$, we have $x_1=\dots = x_n = x$. Denote that the set of extreme points of $A$ by $\E(A)$.
\end{definition}

\begin{definition}[Face of set]
A subset $F\subseteq A$ is said to be a face of $A$ if whenever $x_1,\dots,x_n\in A$, $m_1,\dots,m_n\in \N$, $m= \sum_{i=1}^nm_i$, $mx = \sum_{i=1}^nm_ix_i$ and $x\in F$, then $x_i\in F$ for all $1 \le i \le n$.
\end{definition}

As in vector spaces we have the following.

\begin{prop}[Maximisers are a face]\label{prop max}
Assume that $A\subseteq X$ is a compact convex subset of a topological group. Let $\varphi:X\to \R$ be additive and continuous. Then the set $F_{\varphi} = \big\{x\in A~\big|~ \varphi(x) = \max_{x\in A}\varphi(x)\big\}$ is a compact face of $A$.
\end{prop}

\begin{proof}
First, note that since $A$ is compact and $\varphi$ is continuous, then $F_{\varphi}$ is a nonempty compact set.  Assume that $mx = \sum_{i=1}^nm_ix_i$, where $x_1,\dots,x_n\in A$, $m_1,\dots,m_n\in \N$, $m=\sum_{i=1}^nm_i$, and $x\in F_{\varphi}$. We have
\begin{multline*}
m\max_{x\in A}\varphi(x) = m\varphi(x) = \varphi\left(\sum_{i=1}^nm_ix_i\right)  = \sum_{i=1}^nm_i\varphi(x_i) \le \sum_{i=1}^nm_i\max_{x\in A}\varphi(x) = m\max_{x\in A}\varphi(x).
\end{multline*}
Hence, we must have $\varphi(x_i) = \max_{x\in A}\varphi(x)$, or in other words $x_i\in F_{\varphi}$. This completes the proof.
\end{proof}

\begin{prop}[Existence of extreme points]\label{prop exist}
Assume that $X$ is a semidivisible, connected, locally convex group. Let $C\subseteq X$ be convex and compact. Then $\mathcal E (C)\neq \emptyset$.
\end{prop}

\begin{proof}
If $C$ contains only one point, then since it is convex, we have $\E (C) = C \neq \emptyset$. Assume then that $C$ contains at least two points $x\neq y$. By Proposition~\ref{prop singleton}, $\{x\}$ is convex and since $y\notin \{x\}$, by Theorem~\ref{HB groups strict} there exists $\varphi:X\to \R$ additive such that $\varphi(y)<\varphi(x)$. Thus, by Proposition~\ref{prop max}, $F_{\varphi}$ is a compact face of $C$ and clearly $y\notin F_{\varphi}$. 

Next, repeat the procedure for the set $F_{\varphi}$ instead of for $C$. Altogether, we obtain a sequence of compact faces $\{F_{\varphi}\}_{\varphi}$, which is decreasing. It has a nonempty upper bound, which is the intersection. Choose a minimal elements for the sequence and call it $F$. $F$ is indeed a compact face, since if $mx =\sum_{i=1}^nm_ix_i$, $m_1,\dots,m_n\in \N$, $m=\sum_{i=1}^nm_i$, $x_1,\dots,x_n\in C$ and $x\in F$, then $x\in F_{\varphi}$ for all $\varphi$ in the sequence. Then, since $F_{\varphi}$ is a compact face, we get that $x_i\in F_{\varphi}$ for all $\varphi$. Thus, $x_i\in \bigcap_{\varphi}F_{\varphi}\subseteq F$. 

Thus $F$ is also a face. The compactness of $F$ follows from it being an intersection of compact sets. If $F$ contains more than one point, we can repeat the same procedure and get a contradiction to the maximality of $F$. This completes the proof.
\end{proof}

\begin{thm}[Krein-Milman theorem for groups]\label{km single}
Assume that $X$ is a semidivisible, connected locally convex group. If $C\subseteq X$ is compact and convex, then 
\begin{align*}
C = \overline{\conv}\big(\E(C)\big),
\end{align*}
that is, $C$ is equal to the closed convex hull of it extreme points. 
\end{thm}

\begin{proof}
Let $B$ be the closed convex hull of the extreme points of $A$. We want to show $B=C$. Since $C$ is convex and compact, we clearly have $B\subseteq C$. Assume to the contrary that that there exists $x\in C\setminus B$. $B$ is a compact convex set. Thus, by Theorem~\ref{HB groups strict}, there exists $\varphi:X\to \R$ additive and continuous such that $\sup_{b\in B}\varphi(b)<\varphi(x)$. Then construct the closed face $F_{\varphi}$ as before. We have $B\cap F_{\varphi} = \emptyset$. By Proposition~\ref{prop exist}, $F_{\varphi}$ has an extreme point, which is also an extreme point of $C$. This is a contradiction, and so $B=C$. 
\end{proof}

\begin{remark}
If $X$ is a meet semilattice, so that with $\wedge$ as the monoid operation every element is an idempotent, then for an additive $\varphi$, $\varphi(ny)=\varphi(y)$ for $n \in \N$. This implies that the only finite value of $\varphi$ is zero. Hence, a direct analogue of our results does not hold in this monoid. Note that the extreme points of a convex set are the minimal elements and a Krein-Milman theorem holds in this case  in an appropriate order topology \cite{Pon14}. This can be derived from Stone's lemma for monoids as given in \cite{BG15}. \qede
\end{remark}

\begin{example}[Krein-Milman theorem for the positive hyperbolic group]\label{ex hyper km}
Let $X$ be the positive hyperbolic group, as defined in Example~\ref{ex hyper}. It was noted in Example~\ref{ex hyper}, that this is a connected topological group, which is also locally convex. It was also noted that $X$ is divisible. Let $C\subseteq X$ be a compact subset. Let $\Lambda:\R\to X$ be the map $\theta\mapsto M(\theta)$. Since $\sinh(\cdot)$ is strictly increasing, it follows that $\Lambda$ is a bijection. More specifically, if $x = \left[\begin{array}{cc} a & b \\ b & a\end{array}\right]\in X$, then $\Lambda^{-1}$ is given by
\[\Lambda^{-1}(x) = \mathrm{arcsinh}(b)= \ln  \left(b+\sqrt{b^2+1} \right).\]
In particular, we have $C = \Lambda\left(\Lambda^{-1}(C)\right)$. Also, as was shown in Example~\ref{ex hyper}, if $U\subseteq X$ is open and $M(\theta)\in U$, then there exists $\eps>0$ such that $\Lambda\big((\theta-\eps,\theta+\eps)\big)\subseteq U$, and so $\Lambda$ is a continuous map.  In particular, if $C\subseteq X$ is compact, then $\Lambda^{-1}(C) \subseteq \R$ is compact. Therefore, we have $\conv\left(\Lambda^{-1}(C)\right) \subseteq [\alpha,\beta]$, where $\alpha = \min\big\{\theta~\big|~\theta \in \Lambda^{-1}(C)\big\}$ and $\beta = \max\big\{\theta~\big|~\theta \in \Lambda^{-1}(C)\big\}$. For $M(\theta)\in C$ to be an extreme point, we need that $mM(\theta) = \sum_{i=1}^nm_iM(\theta_i)$ implies $\theta = \theta_i$, $1\le i \le n$. But if $mM(\theta) = \sum_{i=1}^nm_iM(\theta_i)$, then we have $m\theta = \sum_{i=1}^nm_i\theta_i$, and $\theta$ must be an extreme point of $\Lambda^{-1}(C)$. Altogether $\E(C) = \{M(\alpha),M(\beta)\}$, and by Theorem~\ref{km single}, we have that $C = \overline{\conv}\big(\{M(\alpha),M(\beta)\}\big)=M([\alpha,\beta])$, and $C$ is a curve in $\R^4$. \qede
\end{example}

\begin{remark}
The matrices $M(\theta)$ with $\theta \ge 0$ form a partially divisible submonoid, say $H$. Since it is known that a direct product of $p$-semidivisible structures is a similar structure, we have abundant other examples. For example, we may consider any of the groups $X \times \R$ or $X \times X$ or $H \times H$. \qede

\end{remark}

We observe in passing that we can use these extreme point ideas to study the structure of convex cones in topological groups. This allows one use ordered groups  to carefully analysis vector optimisation problems \cite{BG15}.

\subsection{Milman converse theorem in groups}

We should also like to have a group version for the Milman converse theorem \cite{BV10}. 
This turns out to require additional restrictions on the underlying space.
For the converse we first need the following basic property.

\begin{thm}[Theorem 7.4 in \cite{HR79}]\label{thm cover}
Assume that $X$ is a connected Hausdorff topological group, and that $U$ is an open set containing $0$. Then $X = \bigcup_{k=1}^{\infty}kU$.
\end{thm}

\begin{prop}\label{prop cont}
Assume that $X$ is a uniquely divisible, connected locally convex topological group. Let $q = \frac m l \in \Q$, and let $a,b\in X$. Define the function $F=F_{a,b}:\Q\to X$ by $F(q) = x\in X$, where $x$ satisfies $lx = ma+(l-m)b$. Then $F$ is continuous on $\Q$.
\end{prop}

\begin{proof}
First, notice that since $X$ is uniquely divisible, the function $F_{a,b}$ is well defined. Let $U$ be an open convex neighbourhood of $0$. By Theorem~\ref{thm cover}, there exists $k\in \N$ such that $a,b\in kU$. Next, assume that $q_j = \frac{m_j}{l_j} \to q = \frac m l$ and let $x_j = F(q_j)$ and $x=F(q)$. Thus, we have $l_jx_j = m_ja+(l_j-m_j)b$ and $lx = ma+(l-m)b$. Therefore, we also have $ll_j(x-x_j) = l_j\big(ma+(l-m)b\big)-l\big(m_ja+(l_j-m_j)b\big) = \big(l_jm-lm_j\big)a + \big(l_j(l-m)-l(l_j-m_j)\big)b$. Now, since $q_j\to q$, for every $k\in \N$, if $j$ is sufficiently large, we have $|l_jm-lm_j| \le \frac{l\, l_j}k$ and $|\, l_j(l-m)-l(l_j-m_j)| \le \frac{l\, l_j}{2k}$. Thus, $\big(l_jm-lm_j\big)a \in \big(l_jm-lm_j\big)kU \stackrel{(*)}{\subseteq} \frac{l\, l_j k}{2k}U\subseteq \frac{l\, l_j}{2}U$, where in ($*$) we used Proposition~\ref{prop monotone}. Similarly, we can show that for sufficiently large $j\in \N$, we have $\big(l_j(l-m)-l(l_j-m_j)\big)b \in \frac{l\, l_j}{2}U$. 

Altogether, we have that $l \, l_j x \in \frac{l\, l_j}{2}U+ \frac{l\, l_j}{2}U$. Hence, we can write $l\, l_j (x-x_j) = u+u'$, where $2u = \sum_{i=1}^nm_iu_i$, $2u' = \sum_{i=1}^{n'}m_i'u_i'$, $u_i, u_i' \in U$, $\sum_{i=1}^nm_i = \sum_{i=1}^{n'} = l\, l_j$. Hence, $2l\, l_j (x-x_j) = \sum_{i=1}^nm_iu_i+\sum_{i=1}^{n'}m_i'u_i'$ and $\sum_{i=1}^nm_i+\sum_{i=1}^{n'}m_i' = 2l\, l_j$. Since $U$ is convex, it follows that $x-x_j\in U$. Since $U$ is arbitrary, it follows that $F_{a,b}(q_j) \to F_{a,b}(q)$, which proves that $F$ is continuous on $\Q$. 
\end{proof}

\begin{cor}[Convex hull of compact convex sets]\label{prop conv union}
Assume that $X$ is a uniquely divisible, locally convex topological group. Let $A,B\subseteq X$ be  compact convex sets. Then $\conv\big(A\cup B\big)$ is compact. More generally, if $A_1,\dots,A_k\subseteq X$ are compact and convex, then $\conv\left(\bigcup_{j=1}^kA_j\right)$ is compact.
\end{cor}

\begin{proof}
Let $x\in \conv\big(A\cup B\big)$. Then $mx = \sum_{i=1}^nm_ix_i = \sum_{i=1}^{n'}m_ix_i+\sum_{i=n'+1}^nm_ix_i$, where the first sum contains elements from $A$ and the second sum contains elements from $B$. Let $m'=\sum_{i=1}^{n'}m_i$. Since $X$ is divisible and $A$ is convex, there exists $a\in A$ such that $\sum_{i=1}^{n'}m_ix_i = m'a$. Similarly, we have $\sum_{i=n'+1}^nm_ix_i = (m-m')b$, where $b\in B$. Altogether, we have $mx = m'a+(m-m')b$, or in other words $x = F_{a,b}(m'/m)$. Hence we can write
\begin{align*}
\conv\big(A\cup B\big) = \big\{F_{a,b}(q)~\big|~ q\in \Q\cap[0,1], a\in A, b\in B\big\}.
\end{align*}
Since the operations on $X$ are continuous, using Proposition~\ref{prop cont}, the map $(a,b,q)\mapsto F_{a,b}(q)$ is continuous. Since the set $A\times B\times \Q\cap[0,1]$ is compact, it follows that $\conv\big(A\cup B\big)$ is compact. 

To prove the second assertion, use the fact that
\begin{align*}
\conv\left(\bigcup_{j=1}^kA_j\right) \subseteq \conv\left(\conv\left(\bigcup_{j=1}^{k-1}A_j\right)\bigcup A_k\right)
\end{align*}
The result now follows by induction on $k$.
\end{proof}

We are now ready for the promised converse theorem.

\begin{thm}[Milman converse theorem for groups]
Assume that $X$ is a uniquely divisible, locally convex topological group (as holds in a locally convex vector space). Assume that $C\subseteq X$ is a compact set such that $\overline{\conv}(C)$ is compact. Then 
\begin{align*}
\E\left(\,\overline{\conv}(C)\right) \subseteq C.
\end{align*}
\end{thm}

\begin{proof}
Let $x\in \E\left(\,\overline{\conv}(C)\right)$. Let $U$ be a convex neighbourhood of $0$. Since $C$ is compact, there exist finitely many $x_1,\dots,x_n\in C$ such that $C\subseteq \bigcup\big(x_i+\overline U\big)$. Define
\begin{align*}
A_i = \overline{\conv}\left(C\cap (x_i+\overline U)\right).
\end{align*} 
Since $\overline{\conv}(C)$ is assumed compact, it follows that $A_i$ is compact for each $1\le i \le n$. Also, for each $1\le i \le n$, since $C\cap (x_i+\overline U)\subseteq C$, we have $A_i\subseteq \overline{\conv(C)}$, and so $\conv\left(\bigcup_{i=1}^nA_i\right) \subseteq \overline{\conv}(C)$. 
On the other hand, since $C\subseteq \bigcup_{i=1}^n\big(x_i+\overline U\big)$, it follows that $\conv\left(\bigcup_{i=1}^nA_i\right) \supseteq C$.

 Finally, by Proposition~\ref{prop conv union}, we have that $\conv\left(\bigcup_{i=1}^nA_i\right)$ is compact and therefore closed. Altogether, we have $\conv\left(\bigcup_{i=1}^nA_i\right) = \overline{\conv}(C)$. Hence, there exist $m_1,\dots,m_n\in \N\cup\{0\}$ and $x_i\in A_i$ such that $mx=\sum_{i=1}^nm_ix_i$, and $m=\sum_{i=1}^nm_i$. Since $x\in \E\left(\,\overline{\conv}(C)\right)$, we must have $x_1=\dots=x_n=x$. Thus, $x\in A_i\subseteq x_i+\overline{U} \subseteq C+\overline U$. Since $C$ is closed and $U$ is arbitrary, it follows that $x\in C$. This concludes the proof.
\end{proof}

We remark that working in a group shows that many  components of the proof have separate requirements all of which are automatic in a locally convex topological vector space or Banach space. 
\section{Minimax theorem for monoids}\label{sec minimax}

We turn to the proof of a minimax theorem in monoids.

\begin{definition} Let $X$ be a monoid.
A function $f:X\to \R$ is said to be convex-like, if for every $x,y\in X$ and for every $\mu\in[0,1]$, there exist $z\in X$ such that
\begin{align*}
f(z) \le \mu f(x)+(1-\mu)f(y).
\end{align*}
Again, $g$ is said to be concave-like exactly when -$g$ is convex-like.
\end{definition}

The next result most satisfactorily connects convexity of a function in a monoid to abstract convex-likeness.

\begin{prop}\label{prop conv like}
Assume that $X$ is a $p$-semidivisible topological monoid such that for every $x,y\in X$, the set $\conv(\{x,y\})$ is precompact . Assume that $f:X\to [-\infty,\infty)$ is convex and lower semicontinuous. Then $f$ is convex-like. If, instead, we assume that $f$ is concave and upper semicontinuous, then $f$ is concave-like.
\end{prop}

\begin{proof}
Let $x,y\in X$ and $\mu\in[0,1]$. For every $k\in \N$, we can find $m_k\in \N$ and $z_k \in X$ such that $p^kz_k = m_kx+(p^k-y)y$, and $\mu \le \frac {m_k}{p^k} \le \mu +\frac 1 {p^k}$. Such a $z_k$ exists because $X$ is $p$-semidivisible. Now, since $f$ is convex, we have 
\begin{align*}
p^kf(z_k) \le m_kf(x)+(p^k-m_k)f(y),
\end{align*}
and so 
\begin{align*}
f(z_k) \le \frac{m_k}{p^k}f(x)+\left(1-\frac{m_k}{p^k}\right)f(y) \le \left(\mu+\frac 1 {p^k}\right)f(x)+(1-\mu)f(y).
\end{align*}
Since $\conv(\{x,y\})$ is assumed to be precompact, by passing to a subsequence, we may assume without loss of generality that $z_k\to z$. Then, by the semicontinuity of $f$, we get
\begin{align*}
f(z)\le \liminf_{k\to \infty}f(z_k) \le \mu f(x)+(1-\mu)f(y), 
\end{align*}
which completes the proof of the first assertion. The proof of the second assertion follows from replacing $f$ by $-f$. The proof is therefore complete.
\end{proof}

The following theorem was proved in~\cite{BZ86} by easy Lagrange multiplier techniques.

\begin{thm}[Fan's Theorem A in~\cite{BZ86}]
Suppose that $X$ and $Y$ are non-empty sets with $f$ convex-concave-like on $X \times Y$. Suppose that $X$ is compact and $f(\cdot, y)$ is lower semicontinuous on $X$ for each $y \in Y$. Then
\begin{align*}
\min_{x\in X}\sup_{y\in Y}f(x,y) = \sup_{y\in Y}\min_{x\in X}f(x,y).
\end{align*}
\end{thm}

Thus, using Proposition~\ref{prop conv like}, we immediately obtain the following.

\begin{thm}[Minimax formula for partially divisible topological monoids]\label{thm minimax}
Assume that $X$ is a convex and compact subset of a $p$-divisible topological monoid, and $Y$ is a subset of a $q$-divisible topological monoid such that for every $x,y\in Y$, $\conv(\{x,y\})$ is precompact (as holds if, for example, $Y$ is compact and convex). Assume that $f:X\times Y\to \R$ is such that for every $y\in Y$, $f(\cdot,y)$ is convex and lower semicontinuous on $X$, and for every $x\in X$, $f(x,\cdot)$ is concave and upper semicontinuous. Then 
\begin{align}\label{eq minimax}
\min_{x\in X}\sup_{y\in Y}f(x,y) = \sup_{y\in Y}\min_{x\in X}f(x,y).
\end{align}
\end{thm}

\begin{example}[Minimax theorem in the positive hyperbolic group]
Let $X$ be the (positive) hyperbolic group, as defined in Example~\ref{ex hyper}. Let $\Lambda:\R\to X$ be the map defined in Example~\ref{ex hyper km}. Then if $\alpha,\beta\in \R$, we have $\conv\big(\{M(\alpha),M(\beta)\}\big) \subseteq \Lambda\big([\alpha,\beta]\big)$. Since it was shown in Example~\ref{ex hyper km} that $\Lambda$ is continuous, it follows that $\Lambda\big([\alpha,\beta]\big)$ is compact and therefore $\conv\big(\{M(\alpha),M(\beta)\}\big)$ is precompact. Hence, if $C\subseteq X$ is compact and convex and $f:C\times X\to \R$ is such that for every $y\in Y$, $f(\cdot,y)$ is convex and lower semicontinuous on $X$, and for every $x\in X$, $f(x,\cdot)$ is concave and upper semicontinuous, then by Theorem~\ref{thm minimax}, equation~\eqref{eq minimax} holds. 
\qede
\end{example}

\begin{remark} Continuing with the notation of Example \ref{ex hyper km}, we
note that in general we do not have $\conv\big(\{M(\alpha),M(\beta)\}\big) = \Lambda\big([\alpha,\beta]\big)$. For example, if $\alpha=0$ and $\beta=1$, then 
\begin{align*}
\conv\big(\{M(\alpha),M(\beta)\}\big) = \big\{M(\theta)~\big|~\theta\in \Q\cap[0,1]\big\}.
\end{align*}
Thus, when the underlying scalars are incomplete we cannot hope for the convex hull of a pair of points to be anything better than a precompact set. \qede
\end{remark}

\begin{example}[Saddle functions on the positive hyperbolic group]
Let $X$ be the (positive) hyperbolic group, as defined in Example~\ref{ex hyper}. Using~\cite{BG15}, it is known that if $g:\R\to \R$ is convex, then $f:X\to \R$, given by $f(x) = g\left(\Lambda^{-1}(x)\right)$, is also convex. Similarly, if $g:\R^2\to \R$ is convex in the first variable and concave in the second variable, then $f:X\times X\to \R$ given $f(x,y) = g\left(\Lambda^{-1}(x), \Lambda^{-1}(y)\right)$ will have the same saddle properties. Also, since both $\Lambda$ and $\Lambda^{-1}$ are continuous, the continuity properties of $g$ will be inherited by $f$. So if we choose for example $f(x,y) = \left(\Lambda^{-1}(x)\right)^2-\left(\Lambda^{-1}(y)\right)^2$ (i.e., $g(\theta_1,\theta_2) = \theta_1^2-\theta_2^2$), and $C = \Lambda([-1,1])$, then
\begin{align*}
\min_{x\in C}\sup_{y\in X}\Big[\left(\Lambda^{-1}(x)\right)^2-\left(\Lambda^{-1}(y)\right)^2\Big] = \sup_{y\in X}\min_{x\in C}\Big[\left(\Lambda^{-1}(x)\right)^2-\left(\Lambda^{-1}(y)\right)^2\Big].
\end{align*}
Note, however, that not all convex functions on $X$ are of the form explicit form  $g\left(\Lambda^{-1}(x)\right)$, where $g:\R\to \R$ is convex. \qede
\end{example}

\section{Conclusion} We hope that the results we have presented make the case well that it is useful to study locally convex groups and matroids.
In our opinion it both opens up new pathways and sheds new light on old structures.


\begin{thebibliography}{99}

\bib{B15}{article}{
 author={Borwein, J. M.},
    title={A very complicated proof of the minimax theorem},
       date={2015},
    journal ={Minimax Theory and its Applications} ,
    volume={1},
    number={1},
    pages={xx-xx},
   note={Preprint available at \url{https://www.carma.newcastle.edu.au/jon/minimax.pdf}},
}
\bib{BG15}{article}{
   author={Borwein, J. M.},
   author={Giladi, O.},
   title={Convex analysis on groups: a sampler},
   date={2015},
   note={Preprint available at \url{https://www.carma.newcastle.edu.au/jon/ConvOnGroups.pdf}},
}

\bib{BV10}{book}{
   author={Borwein, J. M.},
   author={Vanderwerff, J. D.},
   title={Convex functions: constructions, characterizations and
   counterexamples},
   series={Encyclopedia of Mathematics and its Applications},
   volume={109},
   publisher={Cambridge University Press, Cambridge},
   date={2010},
   pages={x+521},
   isbn={978-0-521-85005-6},
   doi={10.1017/CBO9781139087322},
}

\bib{BZ86}{article}{
   author={Borwein, J. M.},
   author={Zhuang, D.},
   title={On Fan's minimax theorem},
   journal={Math. Programming},
   volume={34},
   date={1986},
   number={2},
   pages={232--234},
   issn={0025-5610},
}

\bib{CT96}{article}{
   author={{\c{C}}akalli, H.},
   author={Thorpe, B.},
   title={Ces\`aro means in topological groups},
   journal={J. Anal.},
   volume={4},
   date={1996},
   pages={9--15},
   issn={0971-3611},
}
	


\bib{Fan53}{article}{
   author={Fan, K.},
   title={Minimax theorems},
   journal={Proc. Nat. Acad. Sci. U. S. A.},
   volume={39},
   date={1953},
   pages={42--47},
   issn={0027-8424},
}


\bib{HR79}{book}{
   author={Hewitt, E.},
   author={Ross, K. A.},
   title={Abstract harmonic analysis. Vol. I},
   series={Grundlehren der Mathematischen Wissenschaften [Fundamental
   Principles of Mathematical Sciences]},
   volume={115},
   edition={2},
   publisher={Springer-Verlag, Berlin-New York},
   date={1979},
   pages={ix+519},
   isbn={3-540-09434-2},
}

\bib{Kau66}{article}{
   author={Kaufman, R.},
   title={Interpolation of additive functionals},
   journal={Studia Math.},
   volume={27},
   date={1966},
   pages={269--272},
   issn={0039-3223},
}

\bib{Kin93}{article}{
   author={Kindler, J.},
   title={Topological intersection theorems},
   journal={Proc. Amer. Math. Soc.},
   volume={117},
   date={1993},
   number={4},
   pages={1003--1011},
   issn={0002-9939},
}

\bib{Kin05}{article}{
   author={Kindler, J.},
   title={A simple proof of Sion's minimax theorem},
   journal={Amer. Math. Monthly},
   volume={112},
   date={2005},
   number={4},
   pages={356--358},
   issn={0002-9890},
}

\bib{Mor63}{article}{
   author={Moreau, J. J.},
   title={Remarques sur les fonctions \`a valeurs dans $[-\infty ,\,+\infty
   ]$ d\'efinies sur un demi-groupe},
   language={French},
   journal={C. R. Acad. Sci. Paris},
   volume={257},
   date={1963},
   pages={3107--3109},
}

\bib{Mor63Book}{book}{
   author={Moreau, J. J.},
   title={Fonctions \`a valeurs dans $[-\infty ,\,+\infty ]$; notions
   alg\'ebriques},
   language={French},
   series={Facult\'e des Sciences de Montpellier, S\'eminaires de
   Math\'ematiques},
   publisher={Universit\'e de Montpellier, Montpellier},
   date={1963},
   pages={v + 37},
}

\bib{Neu28}{article}{
   author={v. Neumann, J.},
   title={Zur Theorie der Gesellschaftsspiele},
   language={German},
   journal={Math. Ann.},
   volume={100},
   date={1928},
   number={1},
   pages={295--320},
   issn={0025-5831},
}

\bib{Par10}{article}{
   author={Park, S.},
   title={Comments on abstract convexity structures on topological spaces},
   journal={Nonlinear Anal.},
   volume={72},
   date={2010},
   number={2},
   pages={549--554},
   issn={0362-546X},
}


\bib{Pon14}{article}{
   author={Poncet, P.},
   title={Convexities on ordered structures have their Krein-Milman theorem},
   journal={J. Convex Anal.},
   volume={21},
   date={2014},
   number={1},
   pages={89--120},
   issn={0944-6532},
}

\bib{Vel93}{book}{
   author={van de Vel, M. L. J.},
   title={Theory of convex structures},
   series={North-Holland Mathematical Library},
   volume={50},
   publisher={North-Holland Publishing Co., Amsterdam},
   date={1993},
   pages={xvi+540},
   isbn={0-444-81505-8},
}

\bib{Sio58}{article}{
   author={Sion, M.},
   title={On general minimax theorems},
   journal={Pacific J. Math.},
   volume={8},
   date={1958},
   pages={171--176},
   issn={0030-8730},
}



\bib{XXC13}{article}{
   author={Xiang, S.},
   author={Xia, S.},
   author={Chen, J.},
   title={KKM lemmas and minimax inequality theorems in abstract convexity
   spaces},
   journal={Fixed Point Theory Appl.},
   date={2013},
   pages={2013:209, 12},
   issn={1687-1812},
}

\end{thebibliography}
\end{document}